\documentclass{article}%
\usepackage{graphicx}
\usepackage{amsmath}
\usepackage{amsfonts}
\usepackage{amssymb}%
\providecommand{\U}[1]{\protect\rule{.1in}{.1in}}
\providecommand{\U}[1]{\protect\rule{.1in}{.1in}}
\providecommand{\U}[1]{\protect\rule{.1in}{.1in}}
\providecommand{\U}[1]{\protect\rule{.1in}{.1in}}
\providecommand{\U}[1]{\protect\rule{.1in}{.1in}}
\providecommand{\U}[1]{\protect\rule{.1in}{.1in}}
\providecommand{\U}[1]{\protect\rule{.1in}{.1in}}
\providecommand{\U}[1]{\protect\rule{.1in}{.1in}}
\providecommand{\U}[1]{\protect\rule{.1in}{.1in}}
\providecommand{\U}[1]{\protect\rule{.1in}{.1in}}
\providecommand{\U}[1]{\protect\rule{.1in}{.1in}}
\providecommand{\U}[1]{\protect\rule{.1in}{.1in}}
\providecommand{\U}[1]{\protect\rule{.1in}{.1in}}
\providecommand{\U}[1]{\protect\rule{.1in}{.1in}}
\providecommand{\U}[1]{\protect\rule{.1in}{.1in}}
\providecommand{\U}[1]{\protect\rule{.1in}{.1in}}
\providecommand{\U}[1]{\protect\rule{.1in}{.1in}}
\providecommand{\U}[1]{\protect\rule{.1in}{.1in}}
\providecommand{\U}[1]{\protect\rule{.1in}{.1in}}
\providecommand{\U}[1]{\protect\rule{.1in}{.1in}}
\providecommand{\U}[1]{\protect\rule{.1in}{.1in}}
\providecommand{\U}[1]{\protect\rule{.1in}{.1in}}
\providecommand{\U}[1]{\protect\rule{.1in}{.1in}}
\providecommand{\U}[1]{\protect\rule{.1in}{.1in}}
\providecommand{\U}[1]{\protect\rule{.1in}{.1in}}
\providecommand{\U}[1]{\protect\rule{.1in}{.1in}}
\providecommand{\U}[1]{\protect\rule{.1in}{.1in}}
\providecommand{\U}[1]{\protect\rule{.1in}{.1in}}
\providecommand{\U}[1]{\protect\rule{.1in}{.1in}}

\newtheorem{theorem}{Theorem}
{}

\newtheorem{condition}{Condition}

\newtheorem{definition}{Definition}
\newtheorem{example}{Example}

\newtheorem{lemma}{Lemma}
{}

\newtheorem{remark}{Remark}

\newenvironment{proof}[1][Proof]{\textbf{#1.} }{\ \rule{0.5em}{0.5em}}

\begin{document}

\title{Spectral Expansion for the Asymptotically Spectral Periodic Differential Operators}
\author{O. A. Veliev\\{\small \ Depart. of Math., Dogus University, }\\{\small Ac\i badem, 34722, Kadik\"{o}y, \ Istanbul, Turkey.}\\\ {\small e-mail: oveliev@dogus.edu.tr}}
\date{}
\maketitle

\begin{abstract}
In this paper we investigate the spectral expansion for the asymptotically
spectral differential operators generated in $L_{2}^{m}(-\infty,\infty)$ by
ordinary differential expression of arbitrary order with periodic
matrix-valued coefficients.



\end{abstract}

\section{Introduction and preliminary Facts}

Let $T^{(m)}(p_{1},P_{2},P_{3},...,P_{n})=:T^{(m)}$ be the differential
operator generated in the space $L_{2}^{m}(-\infty,\infty)$ by the
differential expression
\begin{equation}
l^{(m)}(y)=y^{(n)}+p_{1}I_{m}y^{(n-1)}+P_{2}y^{(n-2)}+P_{3}y^{(n-3)}%
+...+P_{n}y
\end{equation}
and $T_{t}^{(m)}(p_{1},P_{2},P_{3}...,P_{n})=:T_{t}^{(m)}$ for $t\in
\mathbb{C}$ be the differential operator generated in $L_{2}^{m}(0,1)$ by the
same differential expression and the boundary conditions
\begin{equation}
U_{\mathbb{\nu}\text{,}t}(y)=:y^{(\mathbb{\nu})}\left(  1\right)
-e^{it}y^{(\mathbb{\nu})}\left(  0\right)  =0,\text{ }\mathbb{\nu
}=0,1,...,(n-1),
\end{equation}
where $n\geq2,$ $p_{1}$ is $(n-1)$ times continuously differentiable scalar
function, $p_{1}\left(  x+1\right)  =p_{1}\left(  x\right)  ,$ $I_{m}$ is
$m\times m$ unit matrix, $P_{\mathbb{\nu}}$ for $v=2,3,...,n$ are the $m\times
m$ matrix with the complex-valued summable on $[0,1]$ entries, $P_{\mathbb{\nu
}}\left(  x+1\right)  =P_{\mathbb{\nu}}\left(  x\right)  $ and $y=(y_{1}%
,y_{2},...,y_{m})$ is a vector-valued function. Here $L_{2}^{m}(a,b)$ is the
space of the vector-valued functions $f=\left(  f_{1},f_{2},...,f_{m}\right)
$ with the norm $\left\Vert \cdot\right\Vert _{(a,b)}$ and inner product
$(\cdot,\cdot)_{(a,b)}$ defined by%
\[
\left\Vert f\right\Vert _{(a,b)}^{2}=\int_{a}^{b}\left\vert f\left(  x\right)
\right\vert ^{2}dx,\text{ }(f,g)_{(a,b)}=\int_{a}^{b}\left\langle f\left(
x\right)  ,g\left(  x\right)  \right\rangle dx,
\]
where $\left\vert \cdot\right\vert $ and $\left\langle \cdot,\cdot
\right\rangle $ are the norm and inner product in $\mathbb{C}^{m}.$

In this paper we consider the spectral expansion for the operator $T^{(m)}.$
The spectral expansion for the self-adjoint differential operators with
periodic coefficients was constructed by Gelfand [1], Titchmarsh [8] and
Tkachenko [9]. The existence of the spectral singularities and the absence of
the Parseval's equality for the nonself-adjoint operators $T_{t}^{(m)}$ do not
allow us to apply the elegant method of Gelfand (see\ [1]) for the
construction of the spectral expansion for the nonself-adjoint periodic
operators. These situation essentially complicate the construction of\ the
spectral expansion for the nonself-adjoint case.

Note that the spectral singularity of $T^{(m)}$ is a point of its spectrum
$\sigma(T^{(m)})$ in neighborhood on which the projections of $T^{(m)}$ are
not uniformly bounded or equivalently a point $\lambda\in\sigma(T^{(m)})$ is
called a spectral singularity of $T^{(m)}$ if the spectral projections of the
operators $T_{t}^{(m)}$ for $t\in(-\pi,\pi]$ corresponding to the eigenvalues
lying in the small neighborhood of $\lambda$ are not uniformly bounded (see
[10, 11, 2]). Thus here we use the following definition of the spectral singularity.

\begin{definition}
Let $e(t,\gamma)$ be the projection of $T_{t}^{(m)}$ defined by contour
integration of the resolvent of $T_{t}^{(m)}$ over the closed curve $\gamma.$
We say that $\lambda\in\sigma(T^{(m)})$ is a spectral singularity of $T^{(m)}$
if for all $\varepsilon>0,$\ there exists a sequence of closed curves
$\gamma_{n}\subset\{z\in\mathbb{C}:\mid z-\lambda\mid<\varepsilon\}$ such
that
\begin{equation}
\lim_{n\rightarrow\infty}\sup_{t}\parallel e(t,\gamma_{n})\parallel=\infty,
\end{equation}
where $\sup$ is taken over all $t$ for which $\gamma_{n}$ lies in the
resolvent set of $T_{t}^{(m)}$ and $T_{t}^{(m)}$ has a unique simple
eigenvalue inside $\gamma_{n}.$
\end{definition}

This paper can be considered as continuation of the paper [11]. To describe
the scheme of this paper let us introduce some well-known facts and some
results of [11] about eigenvalues (Bloch eigenvalues) and eigenfunction (Bloch
functions) of $T_{t}^{(m)}$ and the problems of the spectral expansion of
\ $T^{(m)}$ which are used essentially.

\textbf{(a) On the Bloch eigenvalues and Bloch functions. }It is well-known
that (see [7, 4] ) the spectrum $\sigma(T^{(m)})$ of $T^{(m)}$ is the union of
the spectra $\sigma(T_{t}^{(m)})$ of $T_{t}^{(m)}$ for $t\in(-\pi,\pi].$
Denote by $T_{t}^{(m)}(0)$ and $T_{t}^{(m)}(C)$ respectively the operator
$T_{t}^{(m)}(0,P_{2},0_{m},0_{m},...,0_{m})$ if $P_{2}(x)=0_{m}$ and
$P_{2}(x)=C,$ where $0_{m}$ is $m\times m$ zero matrix and
\begin{equation}
C=\int_{0}^{1}P_{2}\left(  x\right)  dx.
\end{equation}
It is clear that
\begin{equation}
\varphi_{k,j,t}(x)=e(t)e_{j}e^{i\left(  2\pi k+t\right)  x}%
\end{equation}
\ for $k\in\mathbb{Z},$ \ $j=1,2,...,m,$ where $(e(t))^{-2}=%
{\textstyle\int\nolimits_{0}^{1}}
\mid e^{itx}\mid^{2}dx$ and $e_{1}=(1,0,0,...,0),$ $e_{2}=(0,1,0,...,0),$
$...,e_{m}=(0,0,...,0,1),$ are the normalized eigenfunctions of the operator
$T_{t}^{(m)}(0)$ corresponding to the eigenvalue $\left(  2\pi ki+ti\right)
^{n}.$ One can easily verify that the eigenvalues and normalized
eigenfunctions of $T_{t}^{(m)}(C)$ are%
\begin{equation}
\mu_{k,j}(t)=\left(  2\pi ki+ti\right)  ^{n}+\mu_{j}\left(  2\pi ki+ti\right)
^{n-2}\text{, \ }\Phi_{k,j,t}(x)=e(t)v_{j}e^{i\left(  2\pi k+t\right)  x}%
\end{equation}
for $k\in\mathbb{Z},$ $j=1,2,...,m,$ where $v_{1},v_{2},...,v_{m}$ are the
normalized eigenvectors of the matrix $C$ corresponding to the eigenvalues
$\mu_{1},\mu_{2},...,\mu_{m},$ if the eigenvalues of the matrix $C$ are simple.

In [11] to obtain the asymptotic formulas for $T_{t}^{(m)}$ we took the
operator $T_{t}^{(m)}(C),$ for an unperturbed operator and $T_{t}^{(m)}%
-T_{t}^{(m)}(C)$ for a perturbation and proved the following.

\textbf{Theorem 1.1 of [11] }\textit{Suppose }$p_{1}=0$\textit{ and the
eigenvalues of }$C$\textit{ are simple. }

$(a)$\textit{ The eigenvalues of }$T_{t}^{(m)}$\textit{ consist of }%
$m$\textit{ sequences }$\left\{  \lambda_{k,j}(t):k\in\mathbb{Z}\right\}
$\textit{ for }$j=1,2,...,m$ \textit{satisfying the following, uniform with
respect to }$t$\textit{ in }$Q_{\varepsilon}(n),$ \textit{formula }%
\begin{equation}
\lambda_{k,j}(t)=\left(  2\pi ki+ti\right)  ^{n}+\mu_{j}\left(  2\pi
ki+ti\right)  ^{n-2}+O(k^{n-3}\ln|k|)
\end{equation}
as $k\rightarrow\infty,$ \textit{where }%
\begin{equation}
Q_{\varepsilon}(2\mu)=\{t\in Q:|t-\pi k|>\varepsilon,\forall k\in
\mathbb{Z\}},\text{ }Q_{\varepsilon}(2\mu+1)=Q,\text{ }\varepsilon\in
(0,\frac{\pi}{4})
\end{equation}
\textit{and }$Q$\textit{ is a compact subset of }$\mathbb{C}$\textit{
containing a neighborhood of the interval }$[-\pi,\pi].$\textit{ There exists
a constant }$N(\varepsilon)$\textit{ such that if }$\mid k\mid\geq
N(\varepsilon)$\textit{ and }$t\in Q_{\varepsilon}(n),$\textit{ then }%
$\lambda_{k,j}(t)$\textit{ is a simple eigenvalue of }$T_{t}^{(m)}$\textit{
and the corresponding normalized eigenfunction }$\Psi_{k,j,t}(x)$\textit{
satisfies }%
\begin{equation}
\Psi_{k,j,t}(x)=e(t)v_{j}e^{i\left(  2\pi k+t\right)  x}+O(k^{-1}\ln|k|)
\end{equation}
as $k\rightarrow\infty.$ \textit{This formula is uniform with respect to }%
$t$\textit{ and }$x$\textit{ in }$Q_{\varepsilon}(n)$\textit{ and in }$[0,1].$

$(b)$\textit{ If }$t\in\mathbb{C}(n)$\textit{ then the root functions of
}$T_{t}^{(m)}$\textit{ form a Riesz basis in }$L_{2}^{m}(0,1)$\textit{, where
}$\mathbb{C}(2\mu)=\mathbb{C}\backslash\{\pi k:k\in\mathbb{Z}\}$\textit{,
}$\mathbb{C}(2\mu+1)=\mathbb{C}$\textit{.\ }

$(c)$\textit{ Let }$\left(  T_{t}^{(m)}\right)  ^{\ast}$\textit{ be the
adjoint operator of }$T_{t}^{(m)}$\textit{ and }$X_{k,j,t}$\textit{ be the
eigenfunction of }$\left(  T_{t}^{(m)}\right)  ^{\ast}$\textit{ corresponding
to the eigenvalue }$\overline{\lambda_{k,j}(t)}$\textit{ and satisfying
}$(X_{k,j,t},\Psi_{k,j,t})=1$\textit{, where }$\mid k\mid\geq N(\varepsilon
)$\textit{ and }$t\in Q_{\varepsilon}(n).$\textit{ Then }$X_{k,j,t}%
(x)$\textit{ satisfies the following, uniform with respect to }$t$\textit{ and
}$x$\textit{ in }$Q_{\varepsilon}(n)$\textit{ and in }$[0,1]$\textit{ formula}%
\begin{equation}
X_{k,j,t}(x)=u_{j}(e(t))^{-1}e^{i(2k\pi+\bar{t})x}+O(k^{-1}\ln|k|)
\end{equation}
as $k\rightarrow\infty,$ \textit{where }$u_{j}$\textit{ is the eigenvector of
\ }$C^{\ast}$\textit{ corresponding to }$\overline{\mu_{j}}$\textit{ and
satisfying }$(u_{j},v_{j})=1.$

Note that the formula \ $f(k,t)=O(h(k))$ as $k\rightarrow\infty$ is said to be
uniform with respect to $t$ in a set $E$ if there exist positive constants $N$
and $c$, independent of $t,$ such that $\mid f(k,t)\mid<c\mid h(k)\mid$ for
all $t\in E$ and $\mid k\mid\geq N.$

\begin{remark}
It is well-known that [6] the substitution
\[
Y(x)=\exp\left(  -\frac{1}{n}\int_{0}^{x}p_{1}(t)dt\right)  \widetilde{Y}(x)
\]
reduce the matrix equation
\begin{equation}
Y^{(n)}(x)+p_{1}(x)I_{m}Y^{(n-1)}(x)+P_{2}\left(  x\right)  Y^{(n-2)}%
(x)+...+P_{n}(x)Y=\lambda Y(x)
\end{equation}
to the equation of the form
\begin{equation}
\widetilde{Y}^{(n)}(x)+\widetilde{P}_{2}\left(  x\right)  \widetilde
{Y}^{(n-2)}(x)+\widetilde{P}_{3}\left(  x\right)  \widetilde{Y}^{(n-3)}%
(x)+...+\widetilde{P}_{n}(x)\widetilde{Y}=\lambda\widetilde{Y}(x).
\end{equation}
One can easily verify that
\begin{equation}
\widetilde{P}_{2}\left(  x\right)  =q(x)I_{m}+P_{2}\left(  x\right)  ,
\end{equation}
the eigenvalues $\widetilde{\lambda}_{k,j}(t)$ and eigenfunctions
$\widetilde{\Psi}_{k,j,t}(t)$ of the operator $\widetilde{T}_{t}^{(m)}$
corresponding to (12) satisfies the formula%
\begin{equation}
\widetilde{\lambda}_{k,j}(t+ir)=\lambda_{k,j}(t),\text{ }\widetilde{\Psi
}_{k,j,t+ir}(x)=\Psi_{k,j,t}(t),
\end{equation}
where $q$ is a scalar function $r=\frac{1}{n}\int_{0}^{1}p_{1}(t)dt.$ It
follows from (13) and (4) that the eigenvalues of $\int_{0}^{1}\widetilde
{P}_{2}\left(  x\right)  dx$ are simple whenever the eigenvalues of $C$ are
simple. Therefore the results obtained in \textbf{Theorem 1.1 of [11]} for
$\widetilde{T}_{t}^{(m)}$ continues to hold for $T_{t-ir}^{(m)}.$ Moreover,
\textbf{Theorem 1.1 of [11]} with (14) immediately implies that there exists
$N_{0}$ such that the eigenvalues $\lambda_{k,j}(t)$ of $T_{t}^{(m)}$ for
$\left\vert k\right\vert \geq N_{0}$ and $t\in(-\pi,\pi]$ are simple, the
corresponding functions $\Psi_{k,j,t}(x)$ and $X_{k,j,t}(x)$\textit{ satisfy
the uniform with respect to }$t$\textit{ and }$x$\textit{ in }$(-\pi,\pi
]$\textit{ and in }$[0,1]$\textit{ formulas }
\begin{equation}
\Psi_{k,j,t}(x)=e(t+ir)v_{j}e^{i\left(  2\pi k+t+ir\right)  x}+O(k^{-1}\ln|k|)
\end{equation}
\textit{ and }%
\begin{equation}
X_{k,j,t}(x)=(e(t+ir))^{-1}u_{j}e^{i(2k\pi+t-i\overline{r})x}+O(k^{-1}\ln|k|)
\end{equation}
if one of the following conditions hold
\end{remark}

\begin{condition}
$\ n$ is an odd number and the eigenvalues of $C$ are simple.
\end{condition}

\begin{condition}
$\ n$ is an even number, the eigenvalues of $C$ are simple and
\begin{equation}
\operatorname{Re}\int_{0}^{1}p_{1}\left(  x\right)  dx=\operatorname{Re}%
nr\neq0.
\end{equation}

\end{condition}

Let $Y_{1}(x,\lambda),Y_{2}(x,\lambda),\ldots,Y_{n}(x,\lambda)$ be the
solutions of the matrix equation (11) satisfying $Y_{k}^{(j)}(0,\lambda
)=0_{m}$ for $j\neq k-1$ and $Y_{k}^{(k-1)}(0,\lambda)=I_{m}$. The eigenvalues
of the operator $T_{t}^{(m)}$ are the roots of the characteristic determinant
\begin{equation}
\Delta(\lambda,t)=\det(Y_{j}^{(\nu-1)}(1,\lambda)-e^{it}Y_{j}^{(\nu
-1)}(0,\lambda))_{j,\nu=1}^{n}=
\end{equation}%
\[
e^{inmt}+f_{1}(\lambda)e^{i(nm-1)t}+f_{2}(\lambda)e^{i(nm-2)t}+...+f_{nm-1}%
(\lambda)e^{it}+1
\]
which is a polynomial of $e^{it}$\ with entire coefficients $f_{1}%
(\lambda),f_{2}(\lambda),...$. Therefore the multiple eigenvalues of the
operators $T_{t}^{(m)}$ are the zeros of the resultant $R(\lambda)\equiv
R(\Delta,\Delta^{^{\prime}})$ of the polynomials $\Delta(\lambda,t)$ \ and
$\frac{\partial}{\partial\lambda}\Delta(\lambda,t).$ Since $R(\lambda)$ is
entire function and the large eigenvalues of $T_{t}^{(m)}$ for $t\in(-\pi
,\pi]$ are simple if \textbf{Condition 1} or \textbf{Condition 2} holds (see
the end of Remark 1), there exist at most finite number of multiple
eigenvalues lying in the spectrum $\sigma(T^{(m)})$ of $T^{(m)}$. Denote they
by $a_{1},a_{2},...a_{p.}$ For each $a_{k}$ there are at most $nm$ values of
$t\in(-\pi,\pi]$ satisfying $\Delta(a_{k},t)=0.$ Hence the sets
\begin{equation}
A_{k}=\{t\in(-\pi,\pi]:\Delta(a_{k},t)=0\}\text{ }\And A=\cup_{k=1}^{p}A_{k}%
\end{equation}
are finite and for $t\notin A$ all eigenvalues of $T_{t}^{(m)}$ are simple.

In [11] we proved the following lemma for $\widetilde{T}_{t}^{(m)}$ that
continues to hold for $T_{t}^{(m)}$.

\textbf{Lemma 3.1 of [11]} The eigenvalues $\lambda_{k,j}(t)$ of $T_{t}^{(m)}$
can be numbered as $\lambda_{1}(t),$ $\lambda_{2}(t),...,$ such that for each
$p=:p(k,j)$ the function $\lambda_{p}(t)$ is continuous in $(-\pi,\pi].$
Moreover, if\ $\ |k|\geq N_{0},$ $t\in(-\pi,\pi]$ then
\begin{align}
\lambda_{p(k,j)}(t)  &  =\lambda_{k,j}(t),\\
p(k,j)  &  =2|k|m+j,\text{ }\forall k>0,\nonumber\\
p(k,j)  &  =(2|k|-1)m+j,\text{ }\forall k<0,\nonumber
\end{align}
where $p>N_{1},$ $N_{1}=(2N_{0}-1)m$ and $N_{0}$ is defined in Remark 1.

Thus if \textbf{Condition 1} or \textbf{Condition 2} holds, then for
$t\in\left(  (-\pi,\pi]\backslash A\right)  $ and $k=1,2,...$ the eigenvalues
$\lambda_{k}(t)$ are simple. The corresponding eigenfunctions of $T_{t}^{(m)}$
and $\left(  T_{t}^{(m)}\right)  ^{\ast}$ are denoted by $\Psi_{k,t}$ and
$X_{k,t}$ respectively. Apart from the eigenvalues $\lambda_{k,j}(t),$ where
$|k|\geq N_{0},$ there exist $N_{1}$ eigenvalues of the operator $T_{t}^{(m)}$
denoted by $\lambda_{1}(t),$ $\lambda_{2}(t),...,\lambda_{N_{1}}(t)$ (see p.
12 of [11]). We define $\lambda_{p}(t)$ for $p>N_{1}$ and $t\in\left(
(-\pi,\pi]\backslash A\right)  $ by (20). In this paper both the notations of
\textbf{Theorem 1.1 of [11]} and the notation of \textbf{Lemma 3.1 of [11]
}are used.

\textbf{(b) On the problems of the spectral expansion of} $T^{(m)}.$ By
Gelfand's Lemma (see [1]) for every $f\in L_{2}^{(m)}(-\infty,\infty)$ there
exists $f_{t}(x)$ such that%
\begin{equation}
f(x)=\frac{1}{2\pi}\int\limits_{0}^{2\pi}f_{t}(x)dt,\text{ }\int_{-\infty
}^{\infty}\left\vert f(x)\right\vert ^{2}dx=\frac{1}{2\pi}\int\limits_{0}%
^{2\pi}\int\limits_{0}^{1}\left\vert f_{t}(x)\right\vert ^{2}dxdt
\end{equation}
and%
\begin{equation}
f_{t}(x+1)=e^{it}f_{t}(x),\text{ }\int_{0}^{1}f_{t}(x)\overline{X_{k,t}%
(x)}dx=\int_{-\infty}^{\infty}f(x)\overline{X_{k,t}(x)}dx,
\end{equation}
where $\Psi_{k,t}$ and $X_{k,t}$ are extended to $(-\infty,\infty)$ by%
\begin{equation}
\Psi_{k,t}(x+1)=e^{it}\Psi_{k,t}(x)\text{ }\And X_{k,t}(x+1)=e^{it}X_{k,t}(x).
\end{equation}
Since the system $\left\{  \Psi_{k,t}(x):k\in\mathbb{N}\right\}  $ for
$t\in\left(  (-\pi,\pi]\backslash A\right)  $ form a Reisz basis, we have%

\begin{equation}
f_{t}(x)=\sum_{k\in\mathbb{N}}a_{k}(t)\Psi_{k,t}(x),
\end{equation}
where $\mathbb{N}=\left\{  1,2,...\right\}  ,$
\begin{align}
a_{k}(t)  &  =\int_{0}^{1}f_{t}(x)\overline{X_{k,t}(x)}dx,\\
\text{ }X_{k,t}(x)  &  =\frac{1}{\overline{\alpha_{k}(t)}}\Psi_{k,t}^{\ast
},\text{ }\alpha_{k}(t)=(\Psi_{k,t},\Psi_{k,t}^{\ast}),
\end{align}
and $\Psi_{k,t}^{\ast}$ is the normalized eigenfunctions of $\left(
T_{t}^{(m)}\right)  ^{\ast}.$ Using (24) in (21), we get
\begin{equation}
f(x)=\frac{1}{2\pi}\int\limits_{(-\pi,\pi]}\sum_{k\in\mathbb{N}}a_{k}%
(t)\Psi_{k,t}(x)dt.
\end{equation}

Thus, to construct the spectral expansion we need to consider the following.

\textbf{(i) The existence of the integral over }$(-\pi,\pi]$\textbf{ of the
expression }$a_{k}(t)\Psi_{k,t}$

\textbf{(ii) The investigation of the term by term integration in (27).}

Now we are ready to describe the scheme of this paper. In [5] the spectrality
at $\infty$ of the operator $T^{(m)}$ for the case $m=1$ was investigated in
detail. From the \textbf{Theorem 1.1 of [11]} we immediately obtain that if
one of the \textbf{Condition 1} and \textbf{Condition 2} hold, then the
operator $T^{(m)}$ for general $m$ is an asymptotically spectral operator (see
Theorem 1). According to the definition of the spectrality at $\infty$ given
in [5] for the case $m=1$ (see Definition 3.2 of [5] ), \ and taking into
account that the eigenfunctions of $T_{t}^{(m)}$ for almost all $t$ form a
Riesz basis in $L_{2}^{m}(0,1)$ we give the following definition of the
asymptotic spectrality.

\begin{definition}
The operator $T^{(m)}$ is said to be an asymptotically spectral operator if
there exists a positive constant $M$ such that
\begin{equation}
\sup_{\gamma\in R(M)}(ess\sup_{t\in(-\pi,\pi]}\parallel e(t,\gamma
)\parallel)<\infty,
\end{equation}
where $R(M)$ is the ring consisting of all sets which are the finite union of
the half closed rectangles lying in $\{\lambda\in\mathbb{C}:\mid\lambda
\mid>M\}.$
\end{definition}

The main result of this paper is the construction of the spectral expansion
for each $f\in L_{2}^{m}(-\infty,\infty)$ if one of the \textbf{Condition 1}
and \textbf{Condition 2} holds.\ For this we introduce new concepts as
singular quasimomentum and essential spectral singularities (ESS) (see
Definition 3 in the next section) and consider the effect of these concepts to
the spectral expansion. In particular, for $m=1$ we obtain the spectral
expansion for the operator generated by%
\[
y^{(n)}+p_{1}y^{(n-1)}+p_{2}y^{(n-2)}+p_{3}y^{(n-3)}+...+p_{n}y,
\]
if $n=2\mu+1$ or if $n=2\mu$ and (17) holds. If $n=2\mu$ and $r=0$ then in
general the operator $T^{(1)}$ is not asymptotically spectral operator (see
[12]) and this case will not be considered in this paper. Thus in this paper
we investigate the spectral expansion for asymptotically spectral differential
operators generated by the system of differential expressions with periodic
complex-valued coefficients by using the singular quasimomentum and ESS.

\section{Main Results}

To consider the asymptotic spectrality and spectral expansion we need to study
the series
\begin{equation}
\sum_{k\in\mathbb{N}}(f,X_{k,t})\Psi_{k,t}(x).
\end{equation}
For this let us estimate the remainder
\begin{equation}
R_{l}(x,t)=\sum_{k>l}(f,X_{k,t})\Psi_{k,t}(x),
\end{equation}
where $l\geq N_{1}$ and $N_{1}$ is defined in (20), of the series (29) by
using the notations and results of \textbf{Theorem 1.1 of [11]}.

\begin{lemma}
Suppose one of the \textbf{Condition 1} and \textbf{Condition 2} holds. Let
$J$ be a subset of the set $\mathbb{Z}(s),$ where
\begin{equation}
\mathbb{Z}(s)=:\{(k,j):k\in\mathbb{Z},\mid k\mid\geq s,\text{ }j=1,2,...,m\},
\end{equation}
$s\geq N_{0}$ and $N_{0}$ is defined in Remark 1. There exist a positive
constant $c,$ independent of $t$, $J$ and $f$ such that
\begin{equation}
\parallel%
{\textstyle\sum\limits_{(k,j)\in J}}
(f,X_{k,j,t})\Psi_{k,j,t}\parallel^{2}\leq c\left(
{\textstyle\sum\limits_{(k,j)\in J}}
\mid(f,e_{j}e^{i(2\pi k+t)x})\mid^{2}+\frac{\left\Vert f\right\Vert ^{2}%
}{\sqrt{s}}\right)
\end{equation}
and
\begin{equation}
\parallel%
{\textstyle\sum\limits_{(k,j)\in J}}
(f,X_{k,j,t})\Psi_{k,j,t}(x)\parallel^{2}\leq c\left\Vert f\right\Vert ^{2}%
\end{equation}
for $f\in L_{2}^{m}(0,1)$ and $t\in(-\pi,\pi],$ where $\left\{  e_{i}%
:i=1,2,...,m\right\}  $\ is a standard basis of $\mathbb{C}^{m}$ defined in
(5), $\parallel\cdot\parallel$and $(\cdot,\cdot)$ denotes the norm and inner
product of $L_{2}^{m}(0,1).$
\end{lemma}

\begin{proof}
During the proof of the lemma we denote by $c_{1},c_{2},...$ the positive
constants that do not depend on $t$, $J$ and $f.$ They will be used in the
sense that there exists $c_{i}$ such that the inequality holds. We prove the
lemma when \textbf{Condition 1} holds. The prove of the case when
\textbf{Condition 2} holds is the same. Moreover, formulas (15) and (16) show
that without loss of generality and for simplicity the notations it can be
assumed that $r=0.$ Then $e(t+ir)=1$ for $t\in(-\pi,\pi].$ Now we prove (32)
by showing that the inequalities
\begin{equation}%
{\textstyle\sum\limits_{(k,j)\in J}}
\mid(f,X_{k,j,t})\mid^{2}\leq c_{1}\left(
{\textstyle\sum\limits_{(k,j)\in J}}
\mid(f,e_{j}e^{i(2\pi k+t)x})\mid^{2}+\frac{1}{\sqrt{s}}\parallel
f\parallel^{2}\right)  ,
\end{equation}%
\begin{equation}
\parallel%
{\textstyle\sum\limits_{(k,j)\in J}}
(f,X_{k,j,t})\Psi_{k,j,t}(x)\parallel^{2}\leq c_{2}%
{\textstyle\sum\limits_{(k,j)\in J}}
\mid(f,X_{k,t})\mid^{2}%
\end{equation}
hold. Under the above assumption, it follows from (16) that
\begin{equation}
\mid(f,X_{k,j,t})\mid^{2}\leq c_{3}\left(  \mid(f,u_{j}e^{i(2\pi k+t)x}%
)\mid^{2}+\parallel f\parallel^{2}\left\vert k^{-1}\ln|k|\right\vert
^{2}\right)
\end{equation}
for $\left\vert k\right\vert \geq N_{0}$. Since $\left\{  u_{1},u_{2}%
,...u_{m}\right\}  $ is a basis of $\mathbb{C}^{m}$ and $\left\{
e_{i}e^{i(2\pi k+t)x}:k\in\mathbb{Z};\text{ }i=1,2,...,m\right\}  $ is an
orthonormal basis in $L_{2}^{m}(0,1)$ we have
\begin{equation}%
{\textstyle\sum\limits_{(k,j)\in J}}
\mid(f,u_{j}e^{i(2\pi k+t)x})\mid^{2}\leq c_{4}%
{\textstyle\sum\limits_{(k,j)\in J}}
\mid(f,e_{j}e^{i(2\pi k+t)x})\mid^{2}\leq c_{4}\parallel f\parallel^{2}.
\end{equation}
Thus (34) follows from (36) and the first inequality of (37).

Now we prove (35). For this we use the relations%
\begin{equation}
\Psi_{k,j,t}(x)=v_{j}e^{i(2\pi k+t)x}\text{ }+h_{k,j,t}(x),\text{ }\left\Vert
h_{k,j,t}\right\Vert =O(k^{-1}\ln|k|),
\end{equation}%
\begin{equation}
\left\Vert (f,X_{k,j,t})h_{k,j,t}(x)\right\Vert \leq c_{5}\left(
\mid(f,X_{k,j,t})\mid^{2}+\left\vert k^{-1}\ln|k|\right\vert ^{2}\right)
\end{equation}
obtained from (15) for $\left\vert k\right\vert \geq N_{0}$ under the above
assumption. By (36) and (37)
\begin{equation}%
{\textstyle\sum\limits_{(k,j)\in J}}
\mid(f,X_{k,j,t})\mid^{2}\leq c_{6}\parallel f\parallel^{2}.
\end{equation}
This and (39) imply that the series
\[%
{\textstyle\sum\limits_{(k,j)\in J}}
(f,X_{k,j,t})v_{j}e^{i(2\pi k+t)x}\text{ }\And\text{\ }%
{\textstyle\sum\limits_{(k,j)\in J}}
(f,X_{k,j,t})h_{k,j,t}(x)
\]
converge in the norm of $L_{2}^{m}(0,1)$ and by (38) we have
\begin{equation}
\parallel%
{\textstyle\sum\limits_{(k,j)\in J}}
(f,X_{k,j,t})\Psi_{k,j,t}(x)\parallel^{2}\leq2S_{1}+2S_{2}^{2},
\end{equation}
where
\begin{equation}
S_{2}=\parallel%
{\textstyle\sum\limits_{(k,j)\in J}}
(f,X_{k,j,t})h_{k,j,t}\parallel,
\end{equation}%
\begin{equation}
S_{1}=\parallel%
{\textstyle\sum\limits_{(k,j)\in J}}
(f,X_{k,j,t})v_{j}e^{i(2\pi k+t)x}\text{ }\parallel^{2}\leq c_{7}%
{\textstyle\sum\limits_{(k,j)\in J}}
\mid(f,X_{k,j,t})\mid^{2}.
\end{equation}
\ Now let us estimate $S_{2}.$ It follows from the second equality of (38)
that
\[
S_{2}\leq c_{8}%
{\textstyle\sum\limits_{(k,j)\in J}}
\mid(f,X_{k,j,t})\mid\frac{\ln\mid k\mid}{\mid k\mid}.
\]
Now using the Schwarz inequality for $l_{2}$ we obtain
\begin{equation}
S_{2}^{2}=(%
{\textstyle\sum\limits_{(k,j)\in J}}
\mid(f,X_{k,j,t})\mid^{2})O(s^{-\frac{1}{2}}).
\end{equation}
Therefore (35) follows from (41)-(44). Thus (34) and (35) and hence (32) is
proved. It with the second inequality of (37) yields (33)
\end{proof}

\begin{theorem}
If one of the \textbf{Condition 1} and \textbf{Condition 2} holds then
$T^{(m)}$ is an asymptotically spectral operator.
\end{theorem}

\begin{proof}
Let $M$ be a positive constant such that if $\lambda_{k,j}(t)\in\{\lambda
\in\mathbb{C}:\mid\lambda\mid>M\},$ then $(k,j)\in\mathbb{Z}(N_{0})$ for all
$t\in(-\pi,\pi],$ where $\mathbb{Z}(s)$ is defined in (31). If $\gamma\in
R(M),$ where $R(M)$ is defined in Definition 2, then $\gamma$ encloses finite
number of the eigenvalues of \ $T_{t}^{(m)}.$ Thus, there exists a finite
subset $J$ of $\mathbb{Z}(N_{0})$ such that the eigenvalue $\lambda_{k,j}(t)$
lies inside $\gamma$ if and only if $(k,j)\in J.$ Moreover, by Remark 1,
$\lambda_{k,j}(t)$ for $(k,j)\in J$ is a simple eigenvalue. It is well-known
that the simple eigenvalues are the simple poles of the Green function of
$T_{t}^{(m)}$ and the projection $e(t,\gamma)$ has the form
\begin{equation}
e(t,\gamma)f=\sum_{(k,j)\in J}(f,X_{k,j,t})\Psi_{k,j,t}.
\end{equation}
Therefore the proof of the theorem follows from (33) and Definition 2.
\end{proof}

Now we are ready to consider the \textbf{Spectral Expansion for }$T^{(m)},$
when \textbf{Condition 1} or \textbf{Condition 2} holds. As we noted in the
introduction (see \textbf{(i}) and \textbf{(ii))} for the spectral expansion
we need to consider the integrals of the expression $a_{k}(t)\Psi
_{k,t}(x)=(f_{t},X_{k,t})\Psi_{k,t}(x)$ over $(-\pi,\pi]$ for almost all $x$
and the term by term integration of the series in (27). The functions
$\Psi_{k,t}(x)$ and $X_{k,t}(x)$ for each $x\in\lbrack0,1]$ are defined\ in
$\left(  (-\pi,\pi]\backslash A\right)  $ and in $(-\pi,\pi]$ for $k\leq
N_{1}$ and $k>N_{1}$ respectively, because the corresponding eigenvalue
$\lambda_{k}(t)$ is simple (see Remark 1 and the definition of $N_{1}$ in
(20)). Since $A$\ is a finite set the integrals over $(-\pi,\pi]$ and $\left(
(-\pi,\pi]\backslash A\right)  $ are the same. Now using Lemma 1 we prove the following

\begin{theorem}
If one of the \textbf{Condition 1} and \textbf{Condition 2} are satisfied,
then for each

$f$ $\in L_{2}^{m}(-\infty,\infty)$ the following equality holds
\begin{equation}
f(x)=\frac{1}{2\pi}\int\limits_{(-\pi,\pi]}%
{\displaystyle\sum\limits_{k\leq N_{1}}}
a_{k}(t)\Psi_{k,t}(x)dt+\frac{1}{2\pi}%
{\displaystyle\sum\limits_{k>N_{1}}}
\int\limits_{(-\pi,\pi]}a_{k}(t)\Psi_{k,t}(x)dt.
\end{equation}
The series in (46) converges in the norm of $L_{2}^{m}(a,b)$ for every
$a,b\in\mathbb{R}.$
\end{theorem}

\begin{proof}
As in the proof of Lemma 1 we prove (46) when \textbf{Condition 1} holds and
without loss of generality assume that $r=0.$ The proof of the case when
\textbf{Condition 2} holds is the same. It follows from (21) that if $f$ $\in
L_{2}^{m}(-\infty,\infty)$ then $f_{t}$ $\in L_{2}^{m}(0,1)$ for almost all
$t.$ Therefore using (32) for \ $J=\mathbb{Z}(s)$ and $f=f_{t}$ we obtain
\begin{equation}
\parallel%
{\textstyle\sum\limits_{(k,j)\in\mathbb{N}(s)}}
(f_{t},X_{k,j,t})\Psi_{k,j,t}(x)\parallel^{2}\leq c\left(
{\textstyle\sum\limits_{(k,j)\in\mathbb{N}(s)}}
\mid(f_{t},e_{j}e^{i(2\pi k+t)x})\mid^{2}+\frac{\left\Vert f_{t}\right\Vert
^{2}}{\sqrt{s}}\right)
\end{equation}
for almost all $t.$ By (22)-(24),
\[%
{\textstyle\sum\limits_{(k,j)\in\mathbb{N}(s)}}
(f_{t},X_{k,j,t})\Psi_{k,j,t}(x+1)=e^{it}%
{\textstyle\sum\limits_{(k,j)\in\mathbb{N}(s)}}
(f_{t},X_{k,j,t})\Psi_{k,j,t}(x).
\]
and hence by (47) we have
\begin{equation}
\parallel%
{\textstyle\sum\limits_{(k,j)\in\mathbb{N}(s)}}
(f_{t},X_{k,j,t})\Psi_{k,j,t}\parallel_{(-p,p)}^{2}\leq2pc\left(
{\textstyle\sum\limits_{(k,j)\in\mathbb{N}(s)}}
\mid(f_{t},e_{j}e^{i(2\pi k+t)x})\mid^{2}+\frac{\left\Vert f_{t}\right\Vert
^{2}}{\sqrt{s}}\right)  .
\end{equation}
On the other hand by Parseval equality for $T^{(m)}(0_{m})$,
\[%
{\textstyle\sum\limits_{(k,j)}}
(%
{\textstyle\int\limits_{(-\pi,\pi]}}
\mid(f_{t},e_{j}e^{i(2\pi k+t)x})\mid^{2}dt)=%
{\textstyle\int\limits_{(-\pi,\pi]}}
\left\Vert f_{t}\right\Vert ^{2}dt.
\]
Therefore using (48), (21), (30) and the notation of \textbf{Lemma 3.1 of
[11]} we obtain
\begin{equation}
\int\limits_{(-\pi,\pi]}\int\limits_{(-p,p)}\mid R_{l}(x,t)\mid^{2}%
dxdt\rightarrow0
\end{equation}
as $l\rightarrow\infty.$ Thus by Fubini theorem $R_{l}(x,t)$ for $l\geq N_{1}$
is integrable with respect to $t$ for almost all $x.$ Now using the obvious
inequality%
\[
\mid\int_{(-\pi,\pi]}f(t)dt\mid^{2}\leq2\pi\int_{(-\pi,\pi]}\left\vert
f(t)\right\vert ^{2}dt,
\]
(30) and (49) we obtain%
\begin{equation}
\parallel\int\limits_{(-\pi,\pi]}\sum_{k>l}a_{k}(t)\Psi_{k,t}dt\parallel
_{(-p,p)}^{2}\leq2\pi\int\limits_{(-p,p)}\int\limits_{(-\pi,\pi]}\mid
\sum_{k>l}a_{k}(t)\Psi_{k,t}(x)\mid^{2}dtdx\rightarrow0
\end{equation}
as $l\rightarrow\infty$. Therefore%
\begin{equation}
\sum_{k>l}a_{k}(t)\Psi_{k,t}(x)\text{ }%
\end{equation}
for $l\geq N_{1}$ and hence $a_{k}(t)\Psi_{k,t}(x)$ for $k>N_{1}$ are
integrable and we have
\begin{equation}
\int\limits_{(-\pi,\pi]}\sum_{k>N_{1}}a_{k}(t)\Psi_{k,t}(x)dt=\sum_{k>N_{1}%
}\int\limits_{(-\pi,\pi]}a_{k}(t)\Psi_{k,t}(x)dt,
\end{equation}
where the last series converges in the norm of $L_{2}^{m}(-p,p)$ for every
$p\in\mathbb{N}.$\ The existence of the first integral in (46) follows from
(24) and the integrabilities of $f_{t}(x)$ and (51). Now using (27) we get the
proof of the theorem.
\end{proof}

To obtain the spectral expansion in term of $t$ we need to consider the
existence of
\begin{equation}
\int\nolimits_{(-\pi,\pi]}a_{k}(t)\Psi_{k,t}(x)dt
\end{equation}
for $k\leq N_{1}.$ To consider the existence of (53) we classify the spectral
singularity defined in Definition 1. It is well-known that [6] if $\lambda
_{k}(t)$ is the simple eigenvalues of $T_{t}^{(m)}$ then the projection
$e(t,\gamma)$ defined in Definition 1 has the form
\begin{equation}
e(t,\gamma)f=(f,X_{k,t})\Psi_{k,t}%
\end{equation}
where $\gamma$ contains inside only the eigenvalue $\lambda_{k}(t)$ of
$T_{t}^{(m)}$ . One can readily see%
\begin{equation}
\left\Vert e(t,\gamma)\right\Vert =\frac{1}{\left\vert \alpha_{k}%
(t)\right\vert }.
\end{equation}
Moreover, if $\lambda_{k}(t)$ is a simple eigenvalue then $\left\vert
\alpha_{k}(t)\right\vert \neq0$ and the function$\frac{1}{\left\vert
\alpha_{k}\right\vert }$ is continuous in some neighborhood of $t.$ Therefore,
it follows from (55) and Definitions 1 that the set of spectral singularities
is the subset of the set of the multiple eigenvalues $a_{1},a_{2},...a_{p}$
defined in the introduction (see (19)). Thus there are at most finite number
of spectral singularities denoted by $a_{1},a_{2},...a_{s},$ where $s\leq p,$
if \textbf{Condition 1} or \textbf{Condition 2} holds.

By (55) and Definition 1 for $j\leq s$ there exist $t_{0}\in A_{j},$ where
$A_{j}$ is defined in (19), $k\in\mathbb{N}$ and a sequence $\left\{
t_{n}\right\}  $ such that $\lambda_{k}(t_{0})=a_{j},$ $t_{n}\rightarrow
t_{0}$ and $\frac{1}{\left\vert \alpha_{k}(t_{n})\right\vert }\rightarrow
\infty$ as $n\rightarrow\infty,$ that is, $\frac{1}{\left\vert \alpha
_{k}\right\vert }$ is unbounded in any deleted neighborhood $U$ of $t_{0}.$
$\ $If $\lambda_{k}(t_{0})$ is not a spectral singularity then for some
deleted neighborhood $U$ of $t_{0}$ we have
\begin{equation}
\sup_{t\in U}\frac{1}{\left\vert \alpha_{k}(t)\right\vert }<\infty.
\end{equation}
Thus the boundlessness of $\frac{1}{\alpha_{n}}$ is the characterization of
the spectral singularities. The considerations of the spectral singularities
play the crucial rule for the investigations of the spectrality of $T^{(m)}$.
However our aim is the construction the spectral expansion and by Theorem 2
the spectral expansion is connected with the existence of (53) for all $k.$ If
(56) holds, then arguing as above one can readily see that
\begin{equation}
\int\limits_{U}\int\limits_{(0,1)}\left\vert a_{k}(t)\Psi_{k,t}(x)\right\vert
^{2}dxdt\leq\sup_{t\in U}\frac{1}{\left\vert \alpha_{k}(t)\right\vert }%
\int\limits_{U}\int\limits_{(0,1)}\left\vert f_{t}(x)\right\vert
^{2}dxdt<\infty
\end{equation}
Therefore by Fubini theorem the integral
\begin{equation}
\int\nolimits_{U}\frac{1}{\alpha_{k}(t)}(f,\Psi_{k,t}^{\ast})_{\mathbb{R}}%
\Psi_{k,t}(x)dt
\end{equation}
exists for almost all $x$ if $\lambda_{k}(t_{0})$ is not a spectral
singularity. In general, the converse statement is not true, since $\frac
{1}{\alpha_{n}}$ may have an integrable boundlessness, and the integral (58)
may exists when $\lambda_{k}(t_{0})$ is a spectral singularities. Hence to
construct the spectral expansion for the operator $T^{(m)}$ we need to
introduce a new concept connected with the existence of the integral (58) for
$U\subset(-\pi,\pi]$. Therefore we introduce the following notions for the
construction of the spectral expansion. Note that everywhere the integral over
$U$ denotes the integral over $U\backslash A.$

\begin{definition}
Spectral singularity $\lambda_{0}$ is said to be an essential spectral
singularity (ESS) of $T^{(m)}$ if there exist $k\in\mathbb{N}$ , $t_{0}%
\in(-\pi,\pi]$ and $f\in L_{2}^{m}(-\infty,\infty)$ such that $\lambda
_{0}=\lambda_{k}(t_{0})$ and for each $\varepsilon$ the expression $\frac
{1}{\alpha_{k}(t)}(f,\Psi_{k,t}^{\ast})_{\mathbb{R}}\Psi_{k,t}(x)$ for almost
all $x$ is not integrable on $(t_{0}-\varepsilon,-t_{0}+\varepsilon)$. In this
case $t_{0}$ is called a singular quasimomentum.
\end{definition}

Let $\mathbb{S}$ \ be the set of all $k\in\mathbb{N}$ such that $\Gamma
_{k}=\left\{  \lambda_{k}(t):t\in(-\pi,\pi]\right\}  $ contains ESS. It
follows from Definition 3 that, if $k\notin\mathbb{S}$ then the integral (53)
exist. Therefore Theorem 2 can be written in the form.

\begin{theorem}
If one of the \textbf{Condition 1} and \textbf{Condition 2} are satisfied,
then for each

$f$ $\in L_{2}^{m}(-\infty,\infty)$ the equality
\begin{equation}
f(x)=\frac{1}{2\pi}\int\limits_{(-\pi,\pi]}%
{\displaystyle\sum\limits_{k\in\mathbb{S}}}
a_{k}(t)\Psi_{k,t}(x)dt+\frac{1}{2\pi}%
{\displaystyle\sum\limits_{k\in\mathbb{N}\backslash\mathbb{S}}}
\int\limits_{(-\pi,\pi]}a_{k}(t)\Psi_{k,t}(x)dt
\end{equation}
holds. The series in (59) converges in the norm of $L_{2}^{m}(a,b)$ for every
$a,b\in\mathbb{R}.$ In particular, if $T^{(m)}$ has no ESS, then
\begin{equation}
f(x)=\frac{1}{2\pi}\sum_{k=1}^{\infty}\int_{0}^{2\pi}a_{k}(t)\Psi_{k,t}(x)dt.
\end{equation}

\end{theorem}

Let $\mathbb{E}=:\left\{  t_{1},t_{2},...,t_{s}\right\}  ,$ where $-\pi
<t_{1}<t_{2}<...<t_{s}\leq\pi,$ be the set of all singular quasimomenta of
$T^{(m)}.$ \ Define $I(\delta)$ by
\begin{equation}
I(\delta)=(-\pi,\pi]\backslash%
{\textstyle\bigcup\limits_{j=1}^{s}}
(t_{j}-\delta,t_{j}+\delta),
\end{equation}
where $\delta<\frac{1}{2}\min_{j}\left\{  t_{1}+\pi,t_{j}-t_{j-1}%
:j=2,3,...,s\right\}  ,$ that is, the intervals $(t_{j}-\delta,t_{j}+\delta)$
for $j=2,3,...,s$ are pairwise disjoint.

\begin{theorem}
If one of the \textbf{Condition 1} and \textbf{Condition 2} are satisfied,
then for each

$f\in L_{2}^{m}(-\infty,\infty)$ the following spectral expansion holds
\begin{equation}
f(x)=\frac{1}{2\pi}\lim_{\delta\rightarrow0}\left(  \sum_{k\in\mathbb{S}}%
\int\limits_{I(\delta)}a_{k}(t)\Psi_{k,t}(x)dt\right)  +\frac{1}{2\pi}%
\sum_{k\in\mathbb{N}\backslash\mathbb{S}}\int_{0}^{2\pi}a_{k}(t)\Psi
_{k,t}(x)dt,
\end{equation}
where the series converges in the norm of $L_{2}^{m}(a,b)$ for every
$a,b\in\mathbb{R}.$
\end{theorem}

\begin{proof}
Since
\begin{equation}%
{\displaystyle\sum\limits_{k\in\mathbb{S}}}
a_{k}(t)\Psi_{k,t}(x)
\end{equation}
is integrable over $(-\pi,\pi],$ we have%
\begin{equation}
\lim_{\delta\rightarrow0}\int\limits_{(t_{j}-\delta,t_{j}+\delta)}%
{\displaystyle\sum\limits_{k\in\mathbb{S}}}
a_{k}(t)\Psi_{k,t}(x)dt=0
\end{equation}
for $j=1,2,...,s.$ On the other hand it follows from Definition 3 that
\begin{equation}
\int\limits_{I(\delta)}a_{k}(t)\Psi_{k,t}(x)dt
\end{equation}
exists for all $k\in\mathbb{S}.$ Therefore (62) follows from (59).
\end{proof}

\begin{remark}
By Definition 3 the elements $a_{k}(t)\Psi_{k,t}$ of the set $\left\{
a_{k}(t)\Psi_{k,t}(x):k\in\mathbb{S}\right\}  $ are not integrable on
$(-\pi,\pi].$ The sum of elements of this set is integrable (see (59)) due to
the cancellations of the singular parts of the nonintegrable elements. Thus,
to obtain the spectral expansion (see Theorems 3 and 4) we huddle together the
nonintegrable on $(-\pi,\pi]$ elements. The following example shows that, in
the general case, \ for the considerations of the integrals over $(-\pi,\pi]$
the huddling over $\mathbb{S}$ (see (59) and (62)) is necessary and one can
not divide the set $\mathbb{S}$ into two disjoint subset $\mathbb{S}_{1}$ and
$\mathbb{S}_{2}$ such that the summations
\begin{equation}%
{\displaystyle\sum\limits_{k\in\mathbb{S}_{1}}}
a_{k}(t)\Psi_{k,t}\text{ }\And\text{ }%
{\displaystyle\sum\limits_{k\in\mathbb{S}_{2}}}
a_{k}(t)\Psi_{k,t}%
\end{equation}
are integrable over $(-\pi,\pi]$.
\end{remark}

\begin{example}
For the singular quasimomentum $t_{i}$ denote by $\mathbb{S}(i)$ the set of
all $k$ for which $a_{k}(t)\Psi_{k,t}$ is nonintegrable over the set
$(t_{i}-\delta,t_{i}+\delta).$ It is clear that $\mathbb{S}=\cup
_{i=1,2,...,s}\mathbb{S}(i).$ Let $\mathbb{S}(1)=\left\{  1,2\right\}  ,$
$\mathbb{S}(2)=\left\{  2,3\right\}  ,...,\mathbb{S}(s)=\left\{
s,s+1\right\}  .$ Then $\mathbb{S}=\left\{  1,2,...,s+1\right\}  $ and
$a_{k}(t)\Psi_{k,t}$ is nonintegrable over $(-\pi,\pi]$ for $k=1,2,...,s+1.$
If $s=2$ then it is clear that for any proper subset $\mathbb{S}_{1}$ of
$\mathbb{S}$ the summations in (66), where $\mathbb{S}_{2}=\mathbb{S}%
\backslash\mathbb{S}_{1}$ are not integrable over $(-\pi,\pi].$ This statement
for arbitrary $s$ can be proved by induction method.
\end{example}

\begin{remark}
If we consider the integral over $(-\pi,\pi]$ as sum of integrals over
$I(\delta)$ and

$(t_{j}-\delta,t_{j}+\delta)$ for $j=1,2,...,s,$ where $I(\delta)$ is defined
in (61), then the first integral in (59) can be written in the form
\begin{equation}
\int\limits_{(-\pi,\pi]}%
{\displaystyle\sum\limits_{k\in\mathbb{S}}}
a_{k}(t)\Psi_{k,t}dt=\sum_{k\in\mathbb{S}}\int\limits_{I(\delta)}a_{k}%
(t)\Psi_{k,t}dt+%
{\displaystyle\sum\limits_{i=1}^{s}}
\int\limits_{(t_{i}-\delta,t_{i}+\delta)}%
{\displaystyle\sum\limits_{k\in\mathbb{S}}}
a_{k}(t)\Psi_{k,t}dt.
\end{equation}
By the definition of $\mathbb{S}(i)$ (see Example 1) we have
\begin{equation}
\int\limits_{(t_{j}-\delta,t_{j}+\delta)}%
{\displaystyle\sum\limits_{k\in\mathbb{S}}}
a_{k}(t)\Psi_{k,t}dt=%
{\displaystyle\sum\limits_{k\in\mathbb{S}\backslash\mathbb{S}(i)}}
\int\limits_{(t_{j}-\delta,t_{j}+\delta)}a_{k}(t)\Psi_{k,t}dt+\int
\limits_{(t_{j}-\delta,t_{j}+\delta)}%
{\displaystyle\sum\limits_{k\in\mathbb{S}(i)}}
a_{k}(t)\Psi_{k,t}dt.
\end{equation}
For each singular quasimomentum $t_{i\text{ }}\in\mathbb{E}$ denote by
$\Lambda_{1}(t_{i}),$ $\Lambda_{2}(t_{i}),...,\Lambda_{s_{i}}(t_{i})$ the
different ESS of $T^{(m)}$ lying in $\sigma(T_{t_{i}}^{(m)})$ and put
$\mathbb{S}(i,j)=:\left\{  k\in\mathbb{S(}i\mathbb{)}:\lambda_{k}%
(t_{i})=\Lambda_{j}(t_{i})\right\}  .$ Thus the set of ESS is $\left\{
\Lambda_{j}(t_{i}):\text{ }i=1,2,...,s;\text{ }j=1,2,...,s_{i}\right\}  .$ It
is clear that $\mathbb{S}(i,j)\cap\mathbb{S}(i,v)=\emptyset$ for $j\neq v.$
Hence the set $\mathbb{S(}i\mathbb{)}$ can be divided into pairwise disjoint
subsets $\mathbb{S}(i,j)$ for $j=1,2,...,s_{i}.$\ On the other hand, using
(45) and the well-known argument of the general perturbation theory in finite
dimensional spaces (see Chapter 2 of [3]) one can prove that
\[
e(t,\gamma)f_{t}(x)=\sum_{k=\mathbb{S}(i,j)}(f_{t},X_{k,t})\Psi_{k,t}%
(x),\text{ }\forall t\in(t_{i}-\delta,t_{i}+\delta)\backslash\{t_{j}\}
\]
and $e(t,\gamma)f_{t}(x)$ is integrable in $(t_{i}-\delta,t_{i}+\delta)$ for
almost all $x,$ where $e(t,\gamma)$ \ is the projection defined in Definition
1$,$ $\delta$ is a sufficiently small number and $\gamma$ contains inside only
the eigenvalues $\lambda_{k}(t)$ for $k=\mathbb{S}(i,j)$ of $T_{t}^{(m)}$ for
$t\in(t_{i}-\delta,t_{i}+\delta).$ Therefore the summations over
$\mathbb{S}(i)$ in (68) can be written as the sum of summations over
$\mathbb{S}(i,j)$ for $j=1,2,...,s_{i}:$
\begin{equation}
\int\limits_{(t_{j}-\delta,t_{j}+\delta)}%
{\displaystyle\sum\limits_{k\in\mathbb{S}(i)}}
a_{k}(t)\Psi_{k,t}dt=\sum\limits_{j=1}^{s_{i}}\left(  \int\limits_{(t_{j}%
-\delta,t_{j}+\delta)}%
{\displaystyle\sum\limits_{k\in\mathbb{S}(i,j)}}
a_{k}(t)\Psi_{k,t}dt\right)  .
\end{equation}
We say that the set $\left\{  a_{k}(t)\Psi_{k,t}(x):k\in\mathbb{S}%
(i,j)\right\}  $ is a bundle corresponding to the ESS $\Lambda_{j}(t_{i}).$
For almost all $x$ the total sum $S(x,t)$ of elements of the bundle is an
integrable function in $(t_{i}-\delta,t_{i}+\delta)$. However, each element of
the bundle is nonintegrable over $(t_{i}-\delta,t_{i}+\delta)$ and we huddle
together the nonintegrable elements. Hence \ for the considerations of the
integrals over $(t_{i}-\delta,t_{i}+\delta)$ the huddling over $\mathbb{S}%
(i,j)$ is necessary. Thus in any case the huddling in the spectral expansion
of $T^{(m)}$ is necessary if it has the ESS. Using (67)-(69) one can minimize
the number of terms in the huddling in the spectral expansion. In theorems 3
and 4 to avoid the complicated notations, we prefer to use the integrals over
$(-\pi,\pi]$ and hence the summations (huddling) over $\mathbb{S}$.
\end{remark}

\end{document}